\documentclass[review]{elsarticle}

\journal{Applied Mathematics Letters}
\usepackage{amsmath}
\usepackage{amssymb}
\usepackage{graphics} 
\usepackage{epsfig}
\usepackage{graphicx}
\usepackage{epstopdf}
\usepackage{tikz}
\newtheorem{theorem}{Theorem}[section]
\newtheorem{lemma}[theorem]{Lemma}

\newenvironment{proof}{{\noindent\bf Proof.}}

\begin{document}

\begin{frontmatter}

\title{Uniqueness in determining a convex polygonal source of an elastic body}

\author{Jianli Xiang}
\address{Three Gorges Mathematical Research Center, College of Science, China Three Gorges University, Yichang 443002, People's Republic of China}
\ead{xiangjianli@ctgu.edu.cn}





\begin{abstract}
  In this work, we consider the time-harmonic inverse elastic source problem of a fixed frequency for the Navier equation in two dimensions. We show that a convex polygon can be uniquely determined by a single far field measurement. Our approach relies on the corner singularity analysis of solutions to the inhomogeneous Navier equation with a source term in a sector. This paper also contributes to corner scattering theory for the Navier equation in an non-convex domain.
\end{abstract}


\end{frontmatter}

\section{Introduction}
Inverse problems are very important in science and engineering. Among these, the inverse source problem is an important research topic and has attracted great attention of many researchers over recent years. Its goal is to determine the shape of unknown sources by measuring the scattered wave patterns in the near field or the radiated wave patterns in the far field. Mathematically, the inverse source problem for acoustic \cite{Bao2015,Bao2010,Bleistein1977}, electromagnetic \cite{Albanese2006,Bao2002} and elastic waves \cite{Blasten2019,Griesmaier2018} has been studied widely by many researchers.

The inverse source problem is ill-posed and its solution requires a priori knowledge. In contrast to the time-harmonic inverse scattering problem where there is some control on the incident waves, the inverse source problem of a fixed frequency has received less attention. In 1999, Ikehata \cite{Ikehata1999} used the enclosure method to investigate the two-dimensional problem. He showed that the supporting function of a polygon can be identified by measuring the acoustic field, i.e., the Cauchy data on a remote closed surface. In 2009, Alves, Martins and Roberty \cite{Alves2009} presented a class of source functions where the identification from boundary data is possible and they derived a method to retrieve the source by solving a higher order direct problem. In 2011, Badia and Nara \cite{Badia2011} proved the uniqueness and local stability from the Cauchy data with a single wave number, where the source consists of multiple point sources.

In recent years, the corner scattering method has developed rapidly in inverse scattering theory (\cite{BPS2014,ElHu2015,EH2018}) for justifying the absence of non-scattering energies and non-radiating sources, and then has been extended and used for the shape determination of polyhedral scatterers. In 2018, Bl{\aa}sten \cite{Blasten2018} considered convex polyhedral sources with only one measurement, whose proof was based on an energy identity from the enclosure method and the construction of a new type of planar complex geometrical optics solution. Then in 2019, Bl{\aa}sten and Lin \cite{Blasten2019} extended the single measurement enclosure method technique \cite{Blasten2018} to the elastic setting. In 2020, Hu and Li \cite{Hu2020} provided new insights into inverse source problems with a single far-field pattern in an inhomogeneous background medium. They proved that the gradient of $C^{1,\alpha}$-smooth source terms at corner points can be uniquely identified, in addition to the information on source values at corners and the convex-polygonal support. And then Ma and Hu \cite{Ma2021} addressed the one-wave factorization method to image the convex source support of polygonal type. Inspired by these fruitful researches, we extend the corner scattering technique to study the elastic scattering problem, and have shown that certain penetrable scatterers with rectangular corners scatter every incident wave nontrivially \cite{Xiang2025}.

In this paper we consider convex polyhedral sources with only one measurement. This work is a nontrivial extension of the method proposed in \cite{Hu2020} for the inverse source scattering problem of the Helmholtz equation to solve the inverse source scattering problem of the Navier equation. Clearly, the elastic wave equation is more challenging due to the coexistence of compressional waves and shear waves that propagate at different speeds. Hence novel analysis is required.

The rest of the paper is organized as follows. Section 2 provides a precise description of the source problem and the main result. Section 3 presents the detailed proof and novel analysis, focusing on the uniqueness result of a convex polygonal elastic source with a single far field pattern.

\section{Mathematical Formulation and Main Results}
In this paper, we consider the inverse source scattering problem of time-harmonic elastic waves. Assume that the Lam\'e constants $\lambda$ and $\mu$ satisfying $\mu>0$, $\mu+\lambda>0$. The scattering problem we are dealing with is now modeled by the following Navier equation
\begin{equation*}
\mu\Delta u+(\lambda+\mu)\nabla(\nabla\cdot u)+\omega^{2}u=S, \quad {\rm in} \quad \mathbb{R}^{2}.
\end{equation*}
Here, $u=u^{{\rm in}}+u^{{\rm sc}}$ is the total displacement field which is the superposition of the given incident plane wave $u^{{\rm in}}$ and the scattered wave $u^{{\rm sc}}$. The circular frequency $\omega>0$ and the source term $S$ is supposed to be compactly supported on $\overline{D}$.

By the Helmholtz decomposition theorem, the scattered field $u^{{\rm sc}}$ can be decomposed as $u^{{\rm sc}}=u^{{\rm sc}}_{p}+u^{{\rm sc}}_{s}$, where $u^{{\rm sc}}_{p}$ denotes the compressional wave and $u^{{\rm sc}}_{s}$ denotes the shear wave, $k_{p}$ is the compressional wave number and $k_{s}$ is the shear wave number. They are given by the following forms respectively:
\begin{eqnarray*}
u^{{\rm sc}}_{p}=-\frac{1}{k_{p}^{2}}{\rm grad}\,{\rm div}\, u^{{\rm sc}},\quad k_{p}=\omega\sqrt{\frac{1}{2\mu+\lambda}}, \quad
u^{{\rm sc}}_{s}=\frac{1}{k_{s}^{2}}\overrightarrow{{\rm curl}}\, {\rm curl}\, u^{{\rm sc}},\quad k_{s}=\omega\sqrt{\frac{1}{\mu}},
\end{eqnarray*}
and they satisfy $\triangle u^{{\rm sc}}_{p}+k_{p}^{2}u^{{\rm sc}}_{p}=0$ and $\triangle u^{{\rm sc}}_{s}+k_{s}^{2}u^{{\rm sc}}_{s}=0$. In addition, the Kupradze radiation condition is required to the scattered field $u^{{\rm sc}}$, i.e.
\begin{equation}  \label{rad}
\mathop {\lim }\limits_{r \to \infty }\sqrt[]{r} \Big(\frac{\partial u^{{\rm sc}}_{p}}{\partial r}-ik_{p}u^{{\rm sc}}_{p}\Big)=0,~\quad~
\mathop {\lim }\limits_{r \to \infty }\sqrt[]{r} \Big(\frac{\partial u^{{\rm sc}}_{s}}{\partial r}-ik_{s}u^{{\rm sc}}_{s}\Big)=0, \quad r=|x|.
\end{equation}
And the radiation condition (\ref{rad}) is assumed to hold in all directions $\hat{x}=x/|x|\in\mathbb{S}$, here $\mathbb{S}:=\{x:|x|=1\}$ denotes the unit circle in $\mathbb{R}^{2}$.

It is well known that the radiating solution to the Navier equation has the following asymptotic behavior:
\begin{equation*}
u^{{\rm sc}}(x)=\frac{e^{ik_{p}r}}{\sqrt{r}}u_{p}^{\infty}(\hat{x}) +\frac{e^{ik_{s}r}}{\sqrt{r}} u_{s}^{\infty}(\hat{x}) +\mathcal{O}\left(\frac{1}{r^{3/2}}\right),\quad {\rm as} \quad r\rightarrow+\infty,
\end{equation*}
where $u_{p}^{\infty}(\hat{x})\parallel\hat{x}$ and $u_{s}^{\infty}(\hat{x})\perp\hat{x}$. The functions $u_{p}^{\infty}$, $u_{s}^{\infty}$ are known as compressional and shear far field pattern of $u^{{\rm sc}}$, respectively,
\begin{equation*}
u^{\infty}(\hat{x})=u_{p}^{\infty}(\hat{x})+u_{s}^{\infty}(\hat{x}).
\end{equation*}

Assume that the incident wave is given by
\begin{eqnarray*}
u^{{\rm in}}(x;d,q)=u_{p}^{{\rm in}}(x;d)+u_{s}^{{\rm in}}(x;d,q)=de^{ik_{p}x\cdot d} +qe^{ik_{s}x\cdot d}, \quad d,q\in\mathbb{S},\quad q\perp d,
\end{eqnarray*}
where $d$ is the incident direction. 

\begin{theorem} \label{Main}
Assume that $D\subset\mathbb{R}^{2}$ is a convex polygon and the source term $S(x)$ has non-zero value on every corner of $D$. Then $\partial D$ can be uniquely determined by a single far field pattern $u^{\infty}(\hat{x})$ for all $\hat{x}\in\mathbb{S}$.
\end{theorem}

\section{Preliminary Lemmas and Proof of Main Theorem}
In this section, we will firstly present some lemmas to prepare for the proof of Theorem \ref{Main}, which are also interesting on their own right.

Let $(r,\theta)$ be the polar coordinates of $(x_{1}, x_{2})^{\top}\in\mathbb{R}^{2}$ and $\nu\in\mathbb{R}^{2}$ be the unit normal vector. Denote by $B_{R}(z):=\{x\in\mathbb{R}^{2}:|x-z|<R\}$ the disk centered at $z\in\mathbb{R}^{2}$ with radius $R>0$, and
\begin{equation*}
\mathcal{L}:=\mu\Delta+(\lambda+\mu) \nabla \nabla\cdot, \quad  T:=2\, \mu\, \partial_\nu+\lambda\, \nu\, {\rm div}+\mu\, \nu \times {\rm curl}.
\end{equation*}
\begin{lemma} \label{lem1}
Let $\Gamma\subset\mathbb{R}^2$ be a Lipschitz surface. Then $u_1=u_2$ and $Tu_1=Tu_2$ on $\Gamma$ implies that $\partial_\nu u_1=\partial_\nu u_2$ on $\Gamma$.
\end{lemma}
\begin{proof}
Set $w:=u_1-u_2$. Then we have $w=Tw=0$ on $\Gamma$. Introduce the tangential G\"unter derivative
\begin{eqnarray*}
\mathcal{M}w=\partial_\nu w-\nu\, {\rm div}\,w+\nu \times {\rm curl}\, w.
\end{eqnarray*}
The condition $w=0$ on $\Gamma$ implies that $\mathcal{M}w=0$ on $\Gamma$. Note that
\begin{eqnarray} \label{A.2}
Tw =\mu\, \partial_\nu w+(\lambda+\mu)\, \nu\, {\rm div}\, w
 =(\lambda+2\mu)\,\partial_\nu w+(\lambda+\mu)\, \nu \times {\rm curl}\,w .
\end{eqnarray}
Then we obtain from \eqref{A.2} that
\begin{eqnarray*}
\nu \times Tw=\mu\, \nu \times \partial_\nu w, \quad \nu \cdot Tw=(\lambda+2\mu)\, \nu \cdot \partial_\nu w,
\end{eqnarray*}
respectively. Therefore,
\begin{eqnarray*}
\partial_\nu w=\frac{1}{\mu}(\nu \times Tw) \times \nu+\frac{1}{\lambda+2\mu}(\nu \cdot Tw)\nu=0 \quad {\rm on} \quad \Gamma,
\end{eqnarray*}
i.e., $\partial_\nu u_1=\partial_\nu u_2$ on $\Gamma$.
\end{proof}
\begin{lemma} \label{mainlem}
Suppose the Lam\'e parameters $\mu$ and $\lambda$ are all constants in $B_{R}(O)$. Let $D_{R}:=\{(r,\theta):0<r<R,\,0<\theta<\varphi\}\subset\mathbb{R}^{2}$ be a sector with the opening angle $\varphi\in(0,2\pi)\backslash\{\pi\}$ at the origin $O$. Define $\Gamma_{R}^{+}=\{(r,\theta):0<r<R,\,\theta=0\}$ and $\Gamma_{R}^{-}=\{(r,\theta):0<r<R,\,\theta=\varphi\}$. If the solution pair $u_{\ell}\in [H^{2}(B_{R}(O))]^{2}$ $(\ell=1,2)$ solves the coupling problem
\begin{align*}
&\mathcal{L} u_{\ell}+\omega^{2} u_{\ell}=S_{\ell} \quad {\rm in}~~B_{R}(O), \\
& u_{1}=u_{2},\quad Tu_{1}=Tu_{2}\quad {\rm on}~~\Gamma_{R}^{\pm};
\end{align*}
then $S_{2}(O)=S_{1}(O)$.
\end{lemma}
\begin{proof}
Let $\tau_{\ell}$, $\nu_{\ell}\in\mathbb{R}^{2}$ denote the unit tangential and normal vectors on $\Gamma_{R}^{\pm}$ directed into $D_{R}^{e}=B_{R}(O)\backslash\overline{D_{R}}$. In particular, $\nu_{1}=(0,-1)^{\top}$, $\tau_{1}=(1,0)^{\top}$. Set $w=(w_{1},w_{2})^{\top}=u_{1}-u_{2}$, we have
\begin{align}
&\mathcal{L}w+\omega^{2}w=S_{2}-S_{1} \quad {\rm in} \quad B_{R}(O), \label{eq1} \\
&w=Tw=0 \quad {\rm on} \quad \Gamma_{R}^{\pm}. \label{eq3}
\end{align}
Since the opening angle of $D_{R}$ is not $\pi$, the tangential and normal vectors are linearly independent. Without loss of generality we suppose that $\nu_{1}=c_{1}\tau_{1}+c_{2}\tau_{2}$ with $c_{1}$, $c_{2}\in\mathbb{R}$, $c_{2}\neq0$. Hence,
\begin{eqnarray*}
\partial_{\tau_{2}}=\frac{1}{c_{2}}\partial_{\nu_{1}}-\frac{c_{1}}{c_{2}}\partial_{\tau_{1}}.
\end{eqnarray*}
Then we conclude that for all $l\in\mathbb{N}$,
\begin{eqnarray} \label{eq5}
\nabla_{x}^{l}\subset{\rm span}\{\partial_{\tau_{1}}^{l},\partial_{\tau_{1}}^{l-1}\partial_{\tau_{2}},\partial_{\tau_{1}}^{l-2}\partial_{\tau_{2}}^{2}, \cdots,\partial_{\tau_{1}}^{2}\partial_{\tau_{2}}^{l-2},\partial_{\tau_{1}}\partial_{\tau_{2}}^{l-1},\partial_{\tau_{2}}^{l}\}.
\end{eqnarray}

{\bf Step 1.} We show that $w=0$, $\nabla_{x}w=0$ at $O$.

The result $w(O)=0$ follows immediately from $w=0$ on $\Gamma_{R}^{\pm}$. Applying Lemma \ref{lem1} and the boundary condition (\ref{eq3}) we
get
\begin{eqnarray}  \label{eq8}
\partial_{\tau_{1}}w=\partial_{\nu_{1}}w=0 \quad {\rm on} \quad \Gamma_{R}^{\pm}.
\end{eqnarray}
Since $\nu_{1}$ and $\tau_{1}$ are linearly independent, we have $\nabla_{x}w=0$ at $O$. Therefore, $\nabla_{x} w(O)=0$.

{\bf Step 2.} We show that $\nabla_{x}^{2}w=0, \quad \mathcal{L}w=0, \quad S_{2}-S_{1}=0$ at $O$.

From the boundary condition (\ref{eq8}) we see that
\begin{align*}
& \partial_{\tau_{1}}^{2}w=\partial_{\tau_{1}}\partial_{\nu_{1}}w=0, \quad \partial_{\tau_{2}}^{2}w=\partial_{\tau_{2}}\partial_{\nu_{2}}w=0, \\
& \partial_{\tau_{1}}\partial_{\tau_{2}}w =\frac{1}{c_{2}}\partial_{\tau_{1}}\partial_{\nu_{1}}w -\frac{c_{1}}{c_{2}}\partial_{\tau_{1}}^{2}w=0,
\end{align*}
implying that $\nabla_{x}^{2}w=0$ at $O$ due to the relation (\ref{eq5}). Then $\nabla_{x}^{2}w(O)=0$ and $\mathcal{L}w(O)=0$. We obtain from (\ref{eq1}) that $S_{2}(O)-S_{1}(O)=0$.
\end{proof}
\bigskip

{\bf Proof of Theorem \ref{Main}}. Assume that $u_1^{\infty}=u_2^{\infty}$, then $u_1(x)=u_2(x)$ for all $x\in\mathbb{R}^2 \backslash(\overline{D_1 \cup D_2})$. If $D_1 \neq D_2$, without loss of generality we may assume there exists a corner $O$ of $\partial D_{1}$ such that $O\not\in\overline{D_{2}}$. Notice that this step cannot be achieved if $D_{1}$ and $D_{2}$ are not convex.

Since this corner stays away from $D_{2}$, the function $u_{2}$ satisfies $\mathcal{L} u_{2}+\omega^{2} u_{2}=0$ in $B_{R}(O)$ where $B_{R}(O)\subseteq\mathbb{R}^{2}\backslash \overline{D_{2}}$ contains only one corner of $D_{1}$, while $u_{1}$ fulfills $\mathcal{L} u_{1}+\omega^{2} u_{1}=S_{1}$ with $S_{1}(O)\neq0$. The transmission conditions between $u_{1}$ and $u_{2}$ on $\partial D_{1}\cap B_{R}(O)$ follow from those of $u_{1}$ across $\partial D_{1}$ and the relation $u_1(x)=u_2(x)$ for all $x\in\mathbb{R}^2 \backslash(\overline{D_1 \cup D_2})$. Now, applying Lemma \ref{mainlem} to $u_{1}$, $u_{2}$ in $B_{R}(O)$, we obtain that $S_{1}(O)=0$, contradicting our assumption. Thus $D_1=D_2$.

\section*{Acknowledgments}

This research is supported by the National Natural Science Foundation of China (12301542), the Natural Science Foundation of Yichang (A23-2-027), the Open Research Fund of Hubei Key Laboratory of Mathematical Sciences (Central China Normal University, MPL2025ORG017) and the China Scholarship Council.

\section*{Data availability}

No data was used for the research described in the article.


\begin{thebibliography}{99}
\bibitem{Alves2009} C.J.S. Alves, N.F.M. Martins, N.C. Roberty, Full identification of acoustic sources with multiple frequencies and boundary measurements, Inverse Probl. Imaging 3 (2009) 275-294.
\bibitem{Albanese2006} R. Albanese, P.B. Monk, The inverse source problem for Maxwell's equations, Inverse Probl. 22 (2006) 1023-1035.
\bibitem{Badia2011} E.l. Badia, T. Nara, An inverse source problem for Helmholtz's equation from the Cauchy data with a single wave number, Inverse Probl. 27 (2011) 105001.
\bibitem{Bao2015} G. Bao, P. Li, J. Lin, F. Triki, Inverse scattering problems with multi-frequencies, Inverse Probl. 31 (2015) 093001.
\bibitem{Bao2010} G. Bao, J. Lin, F. Triki, A multi-frequency inverse source problem, J. Diff. Equa. 249 (2010) 3443-3465.
\bibitem{Bao2002} G. Bao, H. Ammari, J.L. Fleming, An inverse source problem for Maxwell's equations in magnetoencephalography, SIAM J. Appl. Math. 62 (2002) 1369-1382.
\bibitem{Blasten2018} E. Bl{\aa}sten, Nonradiating sources and transmission eigenfunctions vanish at corners and edges, SIAM J. Math. Anal. 50 (2018) 6255-6270.
\bibitem{Blasten2019} E. Bl{\aa}sten, Y.H. Lin, Radiating and non-radiating sources in elasticity, Inverse Probl. 35 (2019) 015005.
\bibitem{BPS2014} E. Bl{\aa}sten, L. P\"aiv\"arinta, J. Sylvester, Corners always scatter, Commun. Math. Phys. 331 (2014) 725-753.
\bibitem{Bleistein1977} N. Bleistein, J.K. Cohen, Nonuniqueness in the inverse source problem in acoustics and electromagnetics, J. Math. Phys. 18 (1977) 194-201.
\bibitem{ElHu2015} J. Elschner, G.H. Hu, Corners and edges always scatter, Inverse Probl. 31 (2015) 015003.
\bibitem{EH2018} J. Elschner, G.H. Hu, Acoustic scattering from corners, edges and circular cones, Arch. Rational Mech. Anal. 228 (2018) 653-690.
\bibitem{Griesmaier2018} R. Griesmaier, J. Sylvester, Uncertainty principles for inverse source problems for electromagnetic and elastic waves, Inverse Probl. 34 (2018) 065003.
\bibitem{Hu2020} G.H. Hu, J.Z. Li, Inverse source problems in an inhomogeneous medium with a single far-field pattern, SIAM J. Math. Anal., 52 (5) (2020) 5213-5231.
\bibitem{Ikehata1999} M. Ikehata, Reconstruction of a source domain from the Cauchy data, Inverse Probl. 15 (1999) 637-645.
\bibitem{Ma2021} G.Q. Ma, G.H. Hu, Factorization method with one plane wave: from model-driven and data-driven perspectives, Inverse Probl. 38 (1) (2021) 015003.
\bibitem{Xiang2025} J.L. Xiang, G.H. Hu, Absence of the analytic continuation in elastic transmission eigenfunctions at rectangular corners, in print.
\end{thebibliography}
\end{document}